\documentclass[12pt]{amsart}

 \usepackage{amsfonts,graphicx,amsmath,amsthm,amsfonts,amscd,amssymb,amsmath,latexsym,multicol,mathrsfs,verbatim}
\usepackage{epsfig,url}
\usepackage{flafter}
\usepackage{fancyhdr}
\usepackage{hyperref}

\hypersetup{colorlinks=true, linkcolor=black}

\makeatletter

\def\jobis#1{FF\fi
  \def\predicate{#1}%
  \edef\predicate{\expandafter\strip@prefix\meaning\predicate}%
  \edef\job{\jobname}%
  \ifx\job\predicate
}

\makeatother

\if\jobis{proposal}%
\else
\fi

 \usepackage[matrix, arrow]{xy}

\DeclareMathOperator{\Supp}{Supp}

\DeclareMathOperator{\LCS}{LCS}

\DeclareMathOperator{\codim}{codim}

\DeclareMathOperator{\lct}{lct}

\DeclareMathOperator{\depth}{depth}

 \numberwithin{equation}{subsection}
 \numberwithin{footnote}{subsection}

\newtheorem{thm}{Theorem}[section]

 \newtheorem{cor}[thm]{Corollary}
 \newtheorem{lem}[thm]{Lemma}
 \newtheorem{prop}[thm]{Proposition}

{
    \newtheoremstyle{upright}%
        {8pt plus2pt minus4pt}%
        {8pt plus2pt minus4pt}%
        {\upshape}%
        {}%
        {\bfseries\scshape}%
        {}%
        {1em}%
        {}%
\theoremstyle{upright}

 \newtheorem{defn}[thm]{Definition}

 \newtheorem{rem}[thm]{Remark}

}

 \newcommand{\Q}{\mathbb Q}
 \newcommand{\R}{\mathbb R}
 \newcommand{\E}{\mathbb E}
 \newcommand{\F}{\mathbb F}

 \newcommand{\OO}{\mathcal{O}}

 \newcommand{\bir}{\dashrightarrow}
 \newcommand{\lin}{\sim}
 \newcommand{\num}{\equiv}
 
 \newcommand{\rddown}[1]{\left\lfloor{#1}\right\rfloor} 

\begin{document}

\title{T\lowercase{he }LMMP\lowercase{ for log canonical 3-folds in characteristic $p>5$}}
\author{J\lowercase{oe} W\lowercase{aldron}}
\date{\today}

\begin{abstract}
We prove that one can run the log minimal model program for log canonical $3$-fold pairs in characteristic $p>5$.  In particular we prove the cone theorem, contraction theorem, the existence of flips and the existence of log minimal models for pairs with log divisor numerically equivalent to an effective divisor.  These follow from our main results, which are that certain log minimal models are good.
\end{abstract}

\maketitle

\tableofcontents

\section{Introduction}

All varieties will be over an algebraically closed field $k$ of characteristic $p>5$.

The log minimal model program (LMMP) for klt threefold pairs in characteristic $p>5$ has recently been completed (\cite{keel_basepoint_1999}, \cite{hacon_three_2015}, \cite{birkar_p}, \cite{birkar_waldron}).  Here we prove some results on the LMMP for log canonical threefold pairs in characteristic $p>5$.  Our proofs rely on the LMMP for klt pairs in a crucial way.

Our main results are the following:

\begin{thm}[Good log minimal models 1]\label{lc-basepoint-free}
Let $(X,B)$ be a projective log canonical $3$-fold pair over an algebraically closed field $k$ of characteristic $p>5$ with $\R$-boundary $B$, together with a projective contraction $X\to Z$.  Suppose that $K_X+B$ is nef$/Z$ and big$/Z$.  Then $K_X+B$ is semi-ample$/Z$.
\end{thm}

\begin{thm}[Good log minimal models 2]\label{lc-mori-contraction}
Let $(X,B)$ be a projective log canonical $3$-fold pair over an algebraically closed field $k$ of characteristic $p>5$ with $\R$-boundary $B$, together with a projective contraction $X\to Z$.  Suppose that $A$ is a big and semi-ample$/Z$ $\R$-divisor such that $K_X+B+A$ is nef$/Z$.  Then $K_X+B+A$ is semi-ample$/Z$.
\end{thm}

Similar results in characteristic zero were proven in 
\cite{birkar_lc_flips}, and also in great generality in \cite{fujino_fundamental_2011} using vanishing theorems.   In place of vanishing theorems we use Keel's theorem \cite{keel_basepoint_1999}, which says that a line bundle in positive characteristic is semi-ample if and only if it is semi-ample when restricted to its exceptional locus.  We use the log minimal model program to replace $(X,B)$ with a dlt pair $(Y,B_Y)$ such that the exceptional locus of $K_Y+B_Y$ is contained in the reduced part $\rddown{B_Y}$ of $B_Y$.  We can then obtain semi-ampleness of the restriction to the exceptional locus by using adjunction to a partial normalisation of $\rddown{B_Y}$ and applying abundance for semi-log canonical surfaces \cite{tanaka_slc}. 

Thus one of the main technical results used in the proof is the following:

\begin{thm}\label{restrict}
Let $(Y,B_Y)$ be a projective $\Q$-factorial dlt $3$-fold pair over an algebraically closed field $k$ of characteristic $p>5$ with $\Q$-boundary $B_Y$, such that $K_Y+B_Y$ is nef.  Then $(K_Y+B_Y)|_{\rddown{B_Y}}$ is semi-ample.
\end{thm}

Theorems \ref{lc-basepoint-free} and \ref{lc-mori-contraction} can be used to contract an extremal ray via a projective morphism.  

\begin{cor}[Contraction theorem]\label{contraction-theorem}
Let $(X,B)$ be a projective log canonical $3$-fold pair over an algebraically closed field $k$ of characteristic $p>5$, with $\R$-boundary $B$.  Suppose $R$ is a $K_X+B$-negative extremal ray.  Then there exists a projective contraction $X\to Z$ contracting precisely the curves in $R$.
\end{cor}

In particular this gives projective flipping contractions, and we can also apply Theorem \ref{lc-basepoint-free} to construct flips.  

\begin{cor}[Existence of flips]\label{lcflips}
Let $(X,B)$ be a projective log canonical $3$-fold pair over an algebraically closed field $k$ of characteristic $p>5$, with $\R$-boundary $B$.  Suppose there is an extremal $K_X+B$-flipping contraction $f:X\to Z$.  Then the flip of $f$ exists.
\end{cor}

We also use the ascending chain condition (ACC) for log canonical thresholds to show that any LMMP which begins from an effective pair terminates, as in characteristic zero \cite{birkar_ascending_2007}.

\begin{thm}[Termination for effective pairs]\label{EffTerm}
Let $(X,B)$ be a projective log canonical $3$-fold pair over an algebraically closed field $k$ of characteristic $p>5$, with $\R$-boundary $B$.  Then any sequence of $K_X+B$-flips which are also $M$-flips for some $\R$-Cartier $M\geq0$ terminates.
\end{thm}

In Section \ref{section-cone} we extend the cone theorem to log canonical $3$-folds in characteristic $p>5$.  Note that this gives new information even in the klt case if the variety is not $\Q$-factorial.  

\begin{thm}[Cone theorem]\label{lc-cone}
Let $(X,B)$ be a projective log canonical $3$-fold pair over an algebraically closed field $k$ of characteristic $p>5$, with $\R$-boundary $B$.  Then there exists a countable collection of rational curves $\{C_i\}$ on $X$ such that:
\begin{enumerate}
\item{$\overline{NE}_1(X) = \overline{NE}_1(X) \cap (K_X + B)_{\geq 0} +\sum_i\R_{\geq 0} \cdot [C_i]$.}
\item{$-6 \leq (K_X +B)\cdot C_i < 0$.}
\item{For any ample $\R$-divisor $A$, $(K_X+B+A)\cdot C_i\geq 0$ for all but finitely many $i$.} 
\item{The rays $\{\R_{\geq 0} \cdot [C_i]\}$ do not accumulate in $(K_X+B)_{<0}$.}
\end{enumerate}
\end{thm} 

Putting all of our results together allows us to deduce the following:

\begin{cor}[Log minimal models]\label{log-minimal-models}
Let $(X,B)$ be a projective log canonical $3$-fold pair over an algebraically closed field $k$ of characteristic $p>5$, with $\R$-boundary $B$. Suppose there is a projective contraction $X\to Z$, such that there exists $M\geq 0$ with $K_X+B\num M/Z$.  Then there exists a log minimal model $(Y/Z,B_Y)$ for $(X/Z,B)$, such that $Y\bir X$ does not contract divisors.

In fact this log minimal model can be produced by running a terminating LMMP$/Z$ starting from $(X/Z,B)$.

If in addition $K_X+B$ is big$/Z$ then this log minimal model is good.
\end{cor}

A stronger version of Theorem \ref{lc-basepoint-free} over $\overline{\F}_p$ was proven by Martinelli, Nakamura and Witaszek in \cite{martinelli_nakamura_witaszek}, using different methods.  A version of our Theorem \ref{lc-mori-contraction} over $\overline{\F}_p$ has been obtained independently in \cite{nakamura_witaszek} by the latter two authors using methods similar to our own proof of Theorem \ref{lc-mori-contraction} in Section \ref{section-GLMM}.  A similar method was also used for $4$-folds in characteristic zero in \cite{hashizume_abundance}.

The layout of our paper is as follows.  We first prove the cone theorem (Theorem \ref{lc-cone}) in Section \ref{section-cone}.  We then prove the termination result (Theorem \ref{EffTerm}) in Section \ref{sec:termination}.  Next we come to our main results, proving Theorem \ref{restrict} in Section \ref{sec-restrict} on the way to Theorem \ref{lc-basepoint-free} and Theorem \ref{lc-mori-contraction} in Section \ref{section-GLMM}.  Finally we prove the remaining results, Corollaries \ref{lcflips}, \ref{contraction-theorem} and \ref{log-minimal-models}, in Section \ref{LMMP}.

{\textbf{{Acknowledgements.}}}  The author would like to thank his advisor Caucher Birkar for his support and advice.  He would also like to thank the referee, whose comments have greatly improved the article.  The author is funded by the Engineering and Physical Sciences Research Council (UK).

\section{Preliminaries}

\subsection{Semi-ampleness results in positive characteristic}

This section summarises some special features of positive characteristic which we will use later.  The first is a criterion for semi-ampleness due to Keel.   

\begin{defn}[{\cite[0.1]{keel_basepoint_1999}}]
Let $X$ be a scheme proper over a field, with nef line bundle $L$.  The \emph{exceptional locus} $\E(L)$ of $L$ is defined to be the Zariski closure of the union of all subvarieties $V$ of $X$ such that $L|_V$ is not big, given the reduced scheme structure.
\end{defn}

\begin{prop}[{\cite[1.9]{keel_basepoint_1999}}]\label{KeelE}
Let $X$ be a scheme projective over a field of positive characteristic, with a nef line bundle $L$.  Then $L$ is semi-ample if and only if $L|_{\E(L)}$ is semi-ample. 
\end{prop}

A universal homeomorphism is a morphism of varieties which remains a homeomorphism after arbitrary base change.  For finite morphisms, this is equivalent to a simple condition:

\begin{prop}[{\cite[I.3.7-8]{EGA1},\cite[49]{kollar_variants}}]\label{prop:univ-homeo}
For a finite morphism $f:Y\to X$, the following are equivalent:
\begin{enumerate}
\item{$f$ is a universal homeomorphism.}
\item{$f$ is surjective and injective on geometric points.}
\end{enumerate}
\end{prop}

Any finite universal homeomorphism can be composed with some other finite universal homeomorphism to give a power of the Frobenius morphism. This allows us to move information on line bundles in the reverse direction to usual:

\begin{prop}[{\cite[1.4]{keel_basepoint_1999}}]\label{univhom}
Let $f:X\to Y$ be a finite universal homeomorphism between schemes of finite type over a field of positive characteristic.  Let $L$ be a line bundle on $Y$.  Then $L$ is semi-ample if and only if $f^*L$ is semi-ample.
\end{prop}

\subsection{Demi-normality}

We describe a generalisation of normality particularly suited to use in the LMMP.

\begin{defn}
A scheme satisfies Serre's $S_2$ condition at $x\in X$ if $$\depth_x\OO_{X,x}\geq\min(2,\dim_x \OO_{X,x}).$$
\end{defn}

\begin{prop}[{\cite[5.10-11]{EGA4_2}, \cite[Section 2]{alexeev_limits}}]\label{S2-fication}
If $X$ is a quasi-projective, reduced, equidimensional variety, then the set $U$ where $X$ satisfies $S_2$ is open and $\codim(X-U,X)\geq 2$.  There exists a birational morphism $\phi:Y\to X$, called the $S_2$-fication, such that the following hold:
\begin{enumerate}
\item{For $x\in X$, $X$ satisfies $S_2$ at $x$ if and only if $\phi$ is an isomorphism at $x$.}
\item{$Y$ satisfies $S_2$ at all points.}
\item{$\phi$ is finite and the normalisation of $X$ factors through $\phi$.}
\end{enumerate}
\end{prop}

\begin{rem}
This partial normalisation is also called by various other names in the literature, for example the $S_2$-ization, $Z^{[2]}$-closure and saturation in codimension $2$.
\end{rem}

The best generalisation of normality for the purposes of the LMMP is demi-normality:

\begin{defn}[{\cite{kollar_singularities_2013}}]
A scheme is \emph{demi-normal} if it is $S_2$ and has at worst nodes in codimension $1$.
\end{defn}

For example, in characteristic zero, the reduced part of the boundary of a dlt pair is demi-normal \cite[17.5]{flips_and_abundance}.

\subsection{Singularities of the LMMP}
Here we define the singularity classes most commonly encountered in the LMMP.  For more information see \cite{kollar_singularities_2013}. This work is concerned with extending results known for klt $3$-fold pairs to log canonical $3$-fold pairs.
\begin{defn}\label{def:singularities}
Let $k$ be an algebraically closed field.  A \emph{pair} $(X,B)$ over $k$ consists of a normal variety $X$ over $k$ and an effective $\R$-divisor $B$, called the \emph{boundary}, such that $K_X+B$ is $\R$-Cartier.  Given a birational morphism $\phi:Y\to X$ from another normal variety, we can define $B_Y$ to be the unique $\R$-divisor satisfying $\phi_*B_Y=B$ and $$K_Y+B_Y\lin_{\R}\phi^*(K_X+B).$$ 
For a prime Weil divisor $D$ on such a birational model $Y$, we define the \emph{log discrepancy} of $D$ with respect to $(X, B)$ to be $a(D,X,B):=1-b$ where $b$ is the coefficient of $D$ in $B_Y$.

We say  the pair $(X,B)$ is:
\begin{itemize}
\item{\emph{Kawamata log terminal} (klt) if all $D$ on all birational models of $X$ have $a(D,X,B)>0$.}
\item{\emph{Log canonical} if all $D$ on all birational models of $X$ have $a(D,X,B)\geq 0$.}
\end{itemize}

A \emph{log canonical centre} of $(X,B)$ is the image on $X$ of a divisor of log discrepancy zero with respect to $(X,B)$.  We denote the union of all log canonical centres of a log canonical pair $(X,B)$ by $\mathrm{LCS}(X,B)$.

\begin{itemize}
\item{We say $(X,B)$ is \emph{divisorially log terminal} (dlt) if it is log canonical and there exists a closed subvariety $V\subset X$ such that $(X,B)$ is log smooth outside $V$ and no log canonical centre is contained in $V$.}
\end{itemize}
\end{defn}

If we are in a situation where log resolutions exist, such as for $3$-folds in positive characteristic, to determine if a pair is log canonical or klt it is enough to check discrepancies of just the irreducible Weil divisors on a fixed log resolution of $(X,B)$.

We will need to work with pairs on non-normal surfaces after adjunction, so we need to be able to talk about singularities in the non-normal case.

\begin{defn}[{\cite[5.10]{kollar_singularities_2013}}]
We say a pair $(X,B)$ consisting of a (possibly non-normal) variety $X$ and effective $\Q$-divisor $B$ is \emph{semi-log canonical} or \emph{slc} if:
\begin{itemize}
\item $X$ is a demi-normal scheme with normalisation $\pi:\tilde{X}\to X$.
\item The support of $B$ does not contain any component of the conductor $D$.
\item $K_X+B$ is $\Q$-Cartier.
\item $(\tilde{X},\tilde{D}+\tilde{B})$ is log canonical, where $\tilde{D}$ is the conductor on $\tilde{X}$ and  $\tilde{B}$ is the birational transform of $B$.
\end{itemize}
\end{defn}

The following result was largely proven in \cite{hacon_three_2015} and \cite{kollar_singularities_2013}, and stated in this form as \cite[2.2]{das_hacon}.

\begin{prop}\label{niceintersection}
Let $(X,B)$ be a $\Q$-factorial dlt $3$-fold pair over an algebraically closed field of characteristic $p>5$, with $B = \sum_i D_i +B'$ where $\rddown{B}=\sum_iD_i$.  Then the following hold:
\begin{itemize}
\item{The $s$-codimensional log canonical centres of $(X,B)$ are exactly the irreducible components of the various intersections $D_{i_1}\cap...\cap D_{i_s}$.}
\item{If $i_1,...,i_s$ are distinct, each irreducible component of $D_{i_1}\cap...\cap D_{i_s}$ is normal and of pure codimension $s$.}
\end{itemize}
\end{prop}

\subsection{3-fold LMMP}\label{sub:3-fold-LMMP}

Given a $\Q$-factorial dlt $3$-fold pair $(X,B)$ over an algebraically closed field $k$ of characteristic $p>5$ and a projective contraction $X\to Z$ we may run a $K_X+B$-MMP$/Z$ using the results of \cite{birkar_p} and \cite{birkar_waldron}.  In particular, we can locate extremal rays using the cone theorem \cite[1.1]{birkar_waldron}, contract extremal rays using \cite[1.3]{birkar_waldron} and construct flips using \cite[1.1]{birkar_p}.
We also use the following results:

\begin{prop}[{\cite[5.5]{birkar_p}}]\label{special-term}
Let $(X,B)$ be a projective $\Q$-factorial dlt $3$-fold pair over an algebraically closed field $k$ of characteristic $p>5$ with $\R$-boundary $B$.  Suppose we are given a sequence of $K_X+B$-flips.  Then after finitely many flips, each remaining flip is an isomorphism near $\rddown{B}$.
\end{prop}

\begin{prop}[{\cite[1.6]{birkar_p}}]\label{dlt}
Let $(X,B)$ be a log canonical $3$-fold pair over an algebraically closed field $k$ of characteristic $p>5$ with $\R$-boundary $B$.  Then $(X,B)$ has a crepant $\Q$-factorial dlt model.  In particular this morphism is small over the klt locus of $(X,B)$.
\end{prop}

We will follow the convention that, by definition, if $(Y,B_Y)$ is a log minimal model of $(X,B)$ then $Y\bir X$ does not contract divisors.  Note that this is stronger than the definition used in \cite{birkar_p}, but the difference only occurs when working on non-klt pairs.

\subsection{Polytopes of boundary divisors}

We will use the following result from \cite{birkar_waldron}, proven in characteristic zero in \cite{birkar_on_existence_of_lmm_2_2011}.

\begin{prop}[{\cite[3.8]{birkar_waldron}, \cite[3.2]{birkar_on_existence_of_lmm_2_2011}}]\label{polytope}
Let $X$ be a projective $\Q$-factorial klt variety over $k$ and $V$ be a finite-dimensional rational affine subspace of the space of $\R$-divisors on $X$.  For an $\R$-divisor $D$, if $D=\sum_i d_iD_i$ with $D_i$ distinct and prime, define $||D||=\max\{|d_i|\}$.  Finally define $\mathcal{L}\subset V$ by 
$$
\mathcal{L}=\{\Delta \in V \mid (X,\Delta) ~~\mbox{is log canonical}\}.
$$
$\mathcal{L}$ is a polytope with rational vertices.  Fix $B\in \mathcal{L}$. Then  
there are real numbers $\alpha,\delta>0$, depending only on $(X,B)$ and $V$, such that: 
\begin{enumerate}

\item If $\Delta\in \mathcal{L}$, $||\Delta-B||<\delta$ and  $(K_X+\Delta)\cdot R\le 0$ for an extremal ray $R$,
then $(K_X+B)\cdot R\le 0$.

\item  Let $\{R_t\}_{t\in T}$ be a family of extremal rays of $\overline{NE}(X)$. Then the set 
$$ 
\mathcal{N}_T=\{\Delta \in \mathcal{L} \mid (K_X+\Delta)\cdot R_t\ge 0 ~~\mbox{for any $t\in T$}\}
$$
is a rational polytope.

\item Assume $K_X+B$ is nef, $\Delta\in \mathcal{L}$ satisfies $||\Delta-B||<\delta$, and that $X_i\bir X_{i+1}/Z_i$ is a sequence of $K_X+\Delta$-flips which are $K_X+B$-trivial (where $X=X_1$).  Let $B_i$ (respectively $\Delta_i$) be the birational transform of $B$ (respectively $\Delta$) on $X_i$.  Then if
$(K_{X_i}+\Delta_i)\cdot R\le 0$ for an extremal ray $R$ on some $X_i$, then $(K_{X_i}+B_i)\cdot R= 0$. 
\end{enumerate}
\end{prop} 

\subsection{The LMMP with scaling}

In this subsection we describe a special LMMP, called the LMMP with scaling.  For the purpose of describing the process we will assume that all necessary ingredients exist.   
\begin{defn}\label{def:LMMP-scaling}
Let $(X,B)$ be a log canonical pair and $A\geq 0$ an $\R$-Cartier divisor on $X$.  Suppose also that there is $t_0>0$ such that $(X,B+t_0A)$ is log canonical and $K_X+B+t_0A$ is nef. We describe how to run a $K_X+B$-\emph{MMP with scaling of }$A$.  

Let $\lambda_0=\inf\{t:K_X+B+tA \mathrm{\ is\ nef}\}$, so that $\lambda_0\leq t_0$.  Suppose we can find a $K_X+B$-negative extremal ray $R_0$ which satisfies $(K_X+B+\lambda_0 A)\cdot R_0=0$.  This is the first ray we contract in our LMMP. If the contraction is a Mori fibre contraction we stop, otherwise let $X_1$ be the result of the divisorial contraction or flip.  $K_{X_1}+B_{X_1}+\lambda_0 A_{X_1}$ is also nef, where $B_{X_1}$ and $A_{X_1}$ denote the birational transforms on $X_1$ of $B$ and $A$ respectively.  We define $\lambda_1=\inf\{t:K_{X_1}+B_{X_1}+tA_{X_1} \mathrm{\ is\ nef}\}$.  The next step in our LMMP is chosen to be a $K_{X_1}+B_{X_1}$-negative extremal ray $R_1$ which is $K_{X_1}+B_{X_1}+\lambda_1A_{X_1}$-trivial.  So long as we can always find the appropriate extremal rays, contractions and flips, we can continue this process indefinitely or until the LMMP terminates.
\end{defn}

\begin{lem}\label{lem:LMMP-scaling-num-trivial}
Let $(X,B)$ be a projective $\Q$-factorial log canonical pair with $K_X+B$ nef, and let $E$ be an effective $\R$-Cartier divisor satisfying $E\leq B$ and such that $(X,B-\epsilon E)$ is klt for $\epsilon\leq 1$.  Assume Theorem \ref{lc-cone}, and that all contractions and flips that we need exist.

Suppose that we wish to obtain a log minimal model or Mori fibre space for $(X,B-\epsilon E)$ for some sufficiently small $\epsilon$, where we are free to shrink $\epsilon$.  Then if we attempt to do so by running a $K_X+B-\epsilon E$-MMP with scaling of $E$ we may assume that every extremal ray contracted is $K_X+B$-trivial.

In addition, such an LMMP is also a $K_X+B-\epsilon'E$-MMP with scaling of $E$ for all $\epsilon'<\epsilon$, so if the LMMP terminates it produces a log minimal model or Mori fibre space for $(X,B-\epsilon'E)$ for any $0<\epsilon'\leq\epsilon$.
\end{lem}
\begin{proof}
Choose some $\epsilon<1$ to start the process with.  As we know that $K_X+B$ is nef, $\lambda_0\leq \epsilon$, where $\lambda_0$ is defined as in Definition \ref{def:LMMP-scaling}.  Suppose first that $\lambda_0<\epsilon$.  Then $K_X+B-(\epsilon-\lambda_0)E$ is nef.  As $K_X+B$ is also nef, then $(X,B-\epsilon'E)$ is its own log minimal model for all $\epsilon'\in[0,\epsilon-\lambda_0]$.  This means we can terminate the process by replacing $\epsilon$ with $\epsilon-\lambda_0$.  

Suppose instead that $\lambda_0=\epsilon$.  We claim that there is a $K_X+B-\epsilon E$-negative extremal ray $R_0$ which satisfies $(K_X+B)\cdot R_0=0$.  For if not, by definition of $\lambda_0$ and Theorem \ref{lc-cone}, there exists a sequence of extremal rays $R_i$ and a decreasing sequence of rational numbers $\delta_i\to 0$ which satisfy $(K_X+B-\delta_iE)\cdot R_i<0$ but $(K_X+B)\cdot R_i>0$.  This is impossible by Proposition \ref{polytope}(1).
So there is a $K_X+B$-trivial extremal ray $R_0$ which we contract as the first step of our LMMP.  By linearity this is also $K_X+B-\epsilon'E$-negative for all $\epsilon'\in(0,\epsilon]$.  

If we have not obtained a Mori fibre space then let the result of the flip or divisorial contraction be $X_1$.  Define $\lambda_1$ as in Definition \ref{def:LMMP-scaling}.  Again $\lambda_1\leq \epsilon$.  If $\lambda_1<\epsilon$ then $(X_1,B_{X_1}-\epsilon'E_{X_1})$ is a log minimal model for $(X,B-\epsilon' E)$ for all $\epsilon'\in(0,\epsilon-\lambda_1)$, so we may terminate the process by replacing $\epsilon$ with $\epsilon-\lambda_1$.   On the other hand if $\lambda_1=\epsilon$, we get a $K_{X_1}+B_{X_1}$-trivial contraction.  Continuing this process produces the LMMP described in the statement (even if it does not terminate).  Note that we only needed to replace $\epsilon$ at the end, when we had obtained the model which will give the result.
\end{proof}

\section{Cone Theorem}\label{section-cone}

We will prove the cone theorem for log canonical pairs by passing to a crepant $\Q$-factorial dlt model and using the $\Q$-factorial dlt cone theorem (\cite[Thm 1.1]{birkar_waldron}).  

\begin{lem}\label{trivial-convex-lemma}
Let $f:V\to W$ be a surjective linear map of finite dimensional vector spaces.  Suppose $C_V\subset V$ and $C_W\subset W$ are closed convex cones of maximal dimension and $H\subset W$ is a linear subspace of codimension $1$. Assume:
\begin{itemize}
\item{$f(C_V)=C_W$,}
\item{$C_W\cap H\subset \partial C_W$.}
\end{itemize}
Then $f^{-1}H\cap C_V\subset\partial C_V$ and also $f^{-1}H\cap C_V=f^{-1}(H\cap C_W)\cap C_V$.
\end{lem}
\begin{proof}
First we claim that $$f^{-1}(\partial C_W)\cap C_V\subset\partial C_V.$$  
Take $v\in f^{-1}(\partial C_W)\cap C_V$.  Let $w=f(v)$.  As $w\in \partial C_W$ there is a convergent sequence $w_i\to w$ such that $w_i\not\in C_W$ for all $i$.  $f^{-1}(w_i)$ is an affine space in $V$ which does not intersect $C_V$ (else $w_i\in C_W$), and these affine spaces converge to the affine space $f^{-1}(w)$.  Therefore we can choose $v_i\in f^{-1}(w_i)$ such that $v_i$ converge to $v$.   Thus we have sequence $v_i$ not in $C_V$ converging to $v$ and so $v\in\partial C_V$.

We prove the first claim of the Lemma.  Suppose $v\in f^{-1}H\cap C_V$.  Then $f(v)\in C_W\cap H\subset \partial C_W$.  So $v\in f^{-1}(\partial C_W)$ and also in $C_V$ so it is in $\partial C_V$ by the initial claim.

We prove the second claim.  Suppose $v\in f^{-1}H\cap C_V$.  Let $w=f(v)$.  Then $w\in H$, but also $w\in C_W$ as $f(C_V)=C_W$.  The other inclusion is even more obvious.
\end{proof}

\begin{proof}[Proof of Theorem \ref{lc-cone}]

$(X,B)$ has a crepant $\Q$-factorial dlt model $(Y,B_Y)$ by Proposition \ref{dlt}. This comes with birational morphism $f:Y\to X$, and $B_Y$ is defined to be the dlt boundary satisfying $K_Y+B_Y\lin_\Q \phi^*(K_X+B)$.  There is a surjective linear map of vector spaces $$f_*:N_1(Y)\to N_1(X)$$ which induces a surjection on the pseudo-effective cones $$f_*(\overline{NE}_1(Y))=\overline{NE}_1(X).$$

By the $\Q$-factorial dlt cone theorem, there is a countable collection of rational curves $C_i^Y$ on $Y$ satisfying the requirements of the cone theorem.  In particular,  
$$\overline{NE}_1(Y) = \overline{NE}_1(Y) \cap (K_Y + B_Y)_{\geq 0} +\sum_i\R_{\geq 0} \cdot [C_i^Y].$$
Let $C_i$ be the countable collection of rational curves on $X$ given by letting $C_i$ be $f_*C^Y_i$ with reduced structure.  We claim that these curves satisfy (1).  Suppose instead
$$\overline{NE}_1(X) \neq \overline{NE}_1(X) \cap (K_X + B)_{\geq 0} +\sum_i\R_{\geq 0} \cdot [C_i].$$
Then there is some $\R$-Cartier divisor $D$ which is positive on the right hand side, but non-positive somewhere on $\overline{NE}_1(X)$.  Let $A$ be an ample divisor and $\lambda = \inf\{t:D+tA\mathrm{\ is\ nef}\}$.  Then $D+\lambda A$ is nef but not ample so by Kleiman's criterion it takes value zero somewhere on $\overline{NE}_1(X)\backslash\{0\}$.  By replacing $D$ by $D+\lambda A$ we may assume $D$ is non-negative on $\overline{NE}_1(X)\backslash\{0\}$ but $D_{=0}$ intersects $\overline{NE}_1(X)$ non-trivially.  So $D_{=0}$ cuts out some extremal face $F$ of $\overline{NE}_1(X)$.  By Lemma \ref{trivial-convex-lemma}, $$F_Y:=f_*^{-1}F\cap\overline{NE}_1(Y)=f_*^{-1}D_{=0}\cap\overline{NE}_1(Y)$$ is some non-empty extremal face of $\overline{NE}_1(Y)$, which is $K_Y+B_Y$-negative away from the lower dimensional $f_*^{-1}(0)$.  But any such extremal face contains a $K_Y+B_Y$-negative extremal ray $R$ by the cone theorem \cite[1.1]{birkar_waldron} on $Y$, and $R$ contains one of the $C_i^Y$.  But then $D_{=0}$ contains one of the $C_i$, which contradicts our assumption of inequality.

The inequality $$-6 \leq (K_X +B)\cdot C_i < 0$$ follows directly from the definition of the $C_i$, the $\Q$-factorial dlt cone theorem, the projection formula and the observation $$0<\frac{(K_X+B)\cdot C_i}{(K_X+B)\cdot f_*C_i^Y}\leq 1.$$

Next we show that the rays $R_i=\R_{\geq 0}\cdot[C_i]$ do not accumulate in $(K_X+B)_{<0}$.  Suppose otherwise, so there is some sequence $R_i$ converging to a $K_X+B$-negative ray $R$.  Let $R^Y_i$ be an extremal ray in $\overline{NE}_1(Y)$ satisfying $f_*R^Y_i=R_i$.  Such a ray exists by definition of the $R_i$.  By compactness of the unit ball in $\overline{NE}_1(Y)$, some subsequence of the $R_i^Y$ must converge to a ray $R^Y$.  This must satisfy $f_*R^Y=R$, and so by the projection formula it is $K_Y+B_Y$-negative.  This contradicts the cone theorem for dlt $\Q$-factorial pairs \cite[1.1]{birkar_waldron}.

Finally, let $A$ be an ample $\R$-divisor on $X$.  Suppose there are infinitely many $C_i$ with $(K_X+B+A)\cdot C_i <0$.  By compactness, some subsequence of the corresponding $R_i$ converge to a ray $R$.  This must satisfy $(K_X+B+A)\cdot R\leq 0$, but $R\subset \overline{NE}_1(X)$ so this implies $(K_X+B)\cdot R<0$, which contradicts (4).

\end{proof}

\section{Termination}\label{sec:termination}

In this section we prove Theorem \ref{EffTerm}, using the ideas of the characteristic zero proof in \cite[3.2]{birkar_ascending_2007}.  The next remark is important for the proof.

\begin{rem}[{\cite[3.1]{birkar_ascending_2007}}]\label{ref:termination}
Given a log canonical $3$-fold pair $(X,B)$ we may take a crepant $\Q$-factorial dlt model $\phi:Y\to X$ using Proposition \ref{dlt}.  Let $B_Y$ be the dlt boundary such that $K_Y+B_Y\lin_\R\phi^*(K_X+B)$.  Suppose we have a $K_X+B$-flip $X\bir X^+/Z$.  Then $(X^+,B_{X^+})$ is the unique log canonical model of $(X,B)$ over $Z$.  As remarked in Subsection \ref{sub:3-fold-LMMP} we may run a $K_Y+B_Y$-MMP over $Z$.  If this LMMP terminates (which will follow from Theorem \ref{EffTerm}), say with $Y\bir Y^+$, $(X^+,B_{X^+})$ is also the unique log canonical model for $(Y^+,B_{Y^+})$ over $Z$, so we get a morphism $Y^+\to X^+$ and $(Y^+,B_{Y^+})$ is a crepant $\Q$-factorial dlt model of $(X^+,B_{X^+})$.
\end{rem}

\begin{proof}[Proof of Theorem \ref{EffTerm}]   
Suppose we have a log canonical pair $(X,B)$, and an infinite sequence of $K_X+B$-flips which are also $M$-flips for $M\geq 0$.
$$
\xymatrix@R=10pt@C=0.3pt{
X=X_1 \ar[rd] &   & X_2 \ar[ld] \ar[rd]& & X_3\ar[ld] &...\\
& Z_1 & & Z_2 & & ... \\
}
$$
Let $t^1=\lct(X,B,M)$ and $(X^1,B^1,M^1)=(X,B,M)$.

Let $(Y,\Delta_Y)$ be a crepant $\Q$-factorial dlt model of $(X,B+t^1M)$, with birational morphism $f:Y\to X$.
Let $B_Y$ (resp. $M_Y$) be the birational transform of $B$ (resp. $M$) on $Y$, so that $\Delta_Y=B_Y+t^1M_Y+E$ where $E$ is the reduced exceptional divisor of $f$.
In addition let $0\leq B_Y'\leq B_Y$ (resp. $0\leq M_Y'\leq M_Y$) be the divisors formed as follows: 
\begin{itemize}
\item{If a component of $B_Y$ (resp. $M_Y$) has coefficient $1$ in $B_Y+t^1M_Y$, give it coefficient $0$ in $B_Y'$ (resp. $M_Y'$).}
\item{If a component of $B_Y$ (resp. $M_Y$) has coefficient less than $1$ in $B_Y+t^1M_Y$, give it coefficient in $B_Y'$ (resp. $M_Y'$) equal to its coefficient in $B_Y$ (resp. $M_Y$).}
\end{itemize}
Thus we have $\Delta_Y=B_Y'+t^1M_Y'+\rddown{B_Y+t^1M_Y}+E$.
Run a $K_Y+\Delta_Y$-MMP$/Z_1$. If this terminates it gives us $(Y_2,\Delta_{Y_2})$ which is a crepant $\Q$-factorial dlt model for $(X_2,B_2+t^1M_2)$ by Remark \ref{ref:termination}.  In particular this cannot be isomorphic to $(Y,\Delta_Y)$.  Now repeat the process from $(Y_2,\Delta_{Y_2})$.  Either way we get an infinite sequence of $K_Y+\Delta_Y$-flips. 

By Proposition \ref{special-term} these flips are eventually isomorphisms near $\rddown{\Delta_Y}$, so we may replace $Y$ and the sequence of flips with a truncated version to assume that each flipping locus is disjoint from $\Supp(\rddown{\Delta_Y})$ and also to ensure there are no divisorial contractions.  Each of the $K_Y+\Delta_Y$-flips is now also a $K_Y+B_Y'+t^1M_Y'$-flip and an $M_Y'$-flip.  This is because $K_Y+B_Y'+t^1M_Y'$ (respectively $M_Y'$) only differs from $K_Y+\Delta_Y$ (respectively $M_Y$) on the components of $\Supp(\rddown{\Delta_Y})$.  By assumption the flipping loci are all disjoint from the birational transforms of $\Supp(\rddown{\Delta_Y})$, and so the intersections with the flipping curves are unchanged by changing coefficients of components of $\Supp(\rddown{\Delta_Y})$.

Let $t^2=\lct(Y,B_Y',M_Y')$.  $t^2>t^1$ because $(Y,B_Y'+t^1M_Y')$ is klt by construction.  
Let $(X^2,B^2,M^2,t^2)=(Y,B_Y',M_Y',t^2)$.  We are in the same situation with $X^2$ as we began in with $X$.  Therefore we can inductively create a sequence $t^1<t^2<t^3<...$ of log canonical thresholds for pairs and divisors with coefficients in a finite set.  This contradicts ACC (\cite[1.10]{birkar_p}), so there cannot have been such an infinite sequence of flips. 
\end{proof}

We now use this to extend special termination for $\Q$-factorial dlt pairs (Proposition \ref{special-term}) to general log canonical pairs.

\begin{prop}
Let $(X,B)$ be a projective log canonical $3$-fold pair over $k$.  Suppose we are given a sequence of $K_X+B$-flips, $$X=X_1\bir X_2\bir X_3\bir...$$
Then the flipping locus is eventually disjoint from $\LCS(X,B)$.
\end{prop}
\begin{proof}

Suppose the flipping contractions are $f_i:X_i\to Z_i$ with $X=X_1$.  Let $(Y_1,B_{Y_1})$ be a crepant $\Q$-factorial dlt model of $(X,B)$, which exists by Proposition \ref{dlt}.  Run the $K_{Y_1}+B_{Y_1}$-MMP$/Z_1$ as in Remark \ref{ref:termination} .  
This LMMP is also a $K_{Y_1}+B_{Y_1}+f_1^*A$-MMP for $A$ ample on $Z_1$.  In particular we may choose $A$ sufficiently ample that $K_{Y_1}+B_{Y_1}+f_1^*A$ is big, and so the LMMP terminates by Theorem \ref{EffTerm}.   By Remark \ref{ref:termination} it terminates on $(Y_2,B_{Y_2})$, a crepant $\Q$-factorial dlt model for $(X_2,B_{X_2})$. Continuing to run these LMMPs we get a diagram:
$$
\xymatrix{
Y_1 \ar@{-->}[r] \ar[d] & Y_2\ar[d] \ar@{-->}[r]  & Y_3 \ar[d] \ar@{-->}[r]  & ...\\
X_1 \ar@{-->}[r]       & X_2 \ar@{-->}[r]       & X_3 \ar@{-->}[r]        & ...\\}
$$
where $(Y_{i},B_{Y_i})$ is a $\Q$-factorial dlt model of $(X_i,B_{X_i})$ with birational morphism $g_i:Y_i\to X_i$.  
The top row is a $K_{Y_1}+B_{Y_1}$-MMP, so by special termination for $\Q$-factorial dlt pairs the birational maps $Y_i\bir Y_{i+1}$ are isomorphisms near $\rddown{B_{Y_i}}$ for $i\gg 0$.   Replace the sequences to assume that this holds for all $i$.  We may also assume that there are no divisorial contractions in $Y_i\bir Y_{i-1}$ for each $i$.

Suppose $X_i\bir X_{i+1}$ is not an isomorphism near $\LCS(X_i,B_{X_i})$. 
Let $\phi_i:W\to Y_i$ and $\phi_{i+1}:W\to Y_{i+1}$ be birational morphisms resolving the rational map $Y_i\bir Y_{i+1}$.  
We can assume that $\phi_i$ and $\phi_{i+1}$ are isomorphisms over the locus where $Y_i\bir Y_{i+1}$ is an isomorphism.  In particular they are isomorphisms near $\rddown{B_{Y_i}}$ and $\rddown{B_{Y_{i+1}}}$.
$$D:=\phi_i^*(K_{Y_i}+B_{Y_i})-\phi^*_{i+1}(K_{Y_{i+1}}+B_{Y_{i+1}})$$
is exceptional over $Y_{i}$ and  anti-nef $/Z_i$, so is effective.  $Y_i\bir Y_{i+1}$ being an isomorphism near $\rddown{B_{Y_i}}$ implies that $D$ does not intersect $\phi^{-1}(\rddown{B_{Y_i}})$.
Note that $$D=\phi_i^*g_i^*(K_{X_i}+B_{X_i})-\phi_{i+1}^*g_{i+1}^*(K_{X_{i+1}}+B_{X_{i+1}}).$$

Let $\Gamma_{X_i}$ be a curve in the flipping locus which is not disjoint from $\LCS(X_i,B_{X_i})$.  Let $\Gamma_{Y_i}$ be a curve on $Y_i$ surjective to $\Gamma_{X_i}$ and $\Gamma_W$ a curve on $W$ surjective to $\Gamma_{Y_i}$.   As $\Gamma_{X_i}$ is a flipping curve, $(K_{X_i}+B_{X_i})\cdot \Gamma_{X_i}<0$.  By the projection formula, we see that $D\cdot \Gamma_W<0$, and so $\Gamma_W$ is contained in $\Supp(D)$.

By assumption, $\Gamma_{X_i}$ intersects $\LCS(X_i,B_{X_i})$ and so $\Gamma_W$ intersects $\phi_i^{-1}g_i^{-1}(\LCS(X_i,B_{X_i}))$.  $g_i^{-1}(\LCS(X_i,B_{X_i}))$ consists of $\rddown{B_{Y_i}}$ and possibly finitely many curves which are contracted over $X_i$.  We know that $D$ cannot intersect $\phi_i^{-1}(\rddown{B_{Y_i}})$, so $D$ must be connected to $\phi_i^{-1}(\rddown{B_{Y_i}})$ by a chain of curves, each of which is contracted over $Z_i$.
 Suppose for contradiction that some of these curves are not contained in $\Supp(D)$.  There must be one such curve, $C$ say, which satisfies $D\cdot C>0$, because at least one must intersect $D$ but not be contained in $\Supp(D)$.
  But $C$ is contracted over $Z_i$, which gives a contradiction, for $D$ is anti-nef$/Z_i$.
  

\end{proof}

\section{Restriction}\label{sec-restrict}

In characteristic zero, the reduced boundary $\rddown{B_Y}$ of a dlt pair $(Y,B_Y)$ is $S_2$ and consequently demi-normal.  The proof \cite[17.5]{flips_and_abundance} uses vanishing theorems and may fail in positive characteristic.  Instead we exploit the Frobenius morphism to work on a partial normalisation in place of $\rddown{B_Y}$.

\begin{prop}
Let $(Y,B_Y)$ be a $\Q$-factorial dlt $3$-fold pair over $k$.  Let $\pi:S\to \rddown{B_Y}$ be the $S_2$-fication of $\rddown{B_Y}$.  Then $\pi:S\to \rddown{B_Y}$ is a finite universal homeomorphism from a demi-normal scheme.
\end{prop}

\begin{proof}
$\pi$ is an isomorphism in codimension $1$ because $\rddown{B_Y}$ is reduced and equidimensional.  We show that $\pi$ is injective and surjective on geometric points in order to apply Proposition \ref{prop:univ-homeo}.  It is  surjective because the normalisation is surjective and $\pi$ factors into the normalisation by Proposition \ref{S2-fication}.  

Suppose $P$ is a geometric point on $\rddown{B_Y}$ with more than one pre-image.  We will use what we know of $\rddown{B_Y}$ from Proposition \ref{niceintersection} to reach a contradiction.  Firstly note that each component of $\rddown{B_Y}$ is normal.  This means that we can identify the normalisation of $\rddown{B_Y}$ with the disjoint union of its components. Now observe that $P$ cannot have more than one pre-image in any irreducible component of $S$, because $\pi$ factors into the normalisation of $\rddown{B_Y}$, which is just projection from the disjoint union of the components.

Therefore we may assume $P$ is contained in at least two components $E_1$ and $E_2$ of $\rddown{B_Y}$ and has distinct pre-images $Q_i$ for $i=1,2$, each contained in the component of $S$ corresponding to $E_i$.  $C:=E_1\cap E_2$ is of pure dimension $1$ and its irreducible components are smooth curves by Proposition \ref{niceintersection}.  Let $R_i$ be the unique geometric point in the pre-image of $P$ in $E_i\subset E_1\sqcup E_2$.  As $\pi$ factors through the normalisation $E_1\sqcup E_2\to\rddown{B_Y}$, the image of $R_i$ in $S$ is $Q_i$.  The pre-image of $C$ in $E_1\sqcup E_2$ is supported in pure dimension $1$.  Therefore the pre-image of $C$ on $S$, $C_S$ is also supported in pure dimension $1$.  If $\Gamma$ is an irreducible component of $C$ which contains $P$, the pre-image $\Gamma_i$ of $\Gamma$ in $E_i\subset E_1\sqcup E_2$ is an irreducible curve passing through $R_i$.  There is a unique irreducible curve $\Gamma_S$ which is the pre-image of $\Gamma$ on $S$ because $\pi$ is an isomorphism away from a finite set of points.  So $\Gamma_S$ is the image of both $\Gamma_i$ and it follows that $\Gamma_S$ must contain both $Q_1$ and $Q_2$.  Now let $\Gamma^\nu$ be the normalisation of $\Gamma_S$.  The composition $\Gamma^\nu\to \Gamma$ is an isomorphism of smooth curves, but some point has two geometric pre-images.  We have a contradiction, and so $\pi$ must be a universal homeomorphism.

$S$ is nodal in codimension $1$ because $\rddown{B_Y}$ is nodal in codimension $1$ by \cite[2.32]{kollar_singularities_2013} and $\pi$ is an isomorphism in codimension $1$. $S$ is $S_2$ by definition, so we conclude that $S$ is demi-normal.
\end{proof}

\begin{proof}[Proof of Theorem \ref{restrict}]
The conditions (1)-(6) of \cite[4.2]{kollar_singularities_2013} for adjunction to $S$ are satisfied.  The only condition which is not obvious is (5), which holds because $S$ is demi-normal \cite[5.1]{kollar_singularities_2013}. 
Thus there is a different, $B_S$, satisfying 
$$\pi^*(K_Y+B_Y)\lin_\Q K_S+B_S.$$
We can check discrepancies to ensure this pair is slc by pulling back to the normalisation, and applying adjunction to the individual normal components of $\rddown{B_Y}$.
By Tanaka's abundance for slc surfaces \cite{tanaka_slc}, $K_S+B_S$ is semi-ample.  Now because $S\to \rddown{B_Y}$ is a finite universal homeomorphism, by Proposition \ref{univhom}, $(K_Y+B_Y)|_{\rddown{B_Y}}$ is semi-ample.
\end{proof}

\section{Good log minimal models}\label{section-GLMM}

\subsection{Big log divisors}

In this subsection we prove Theorem \ref{lc-basepoint-free}, which will allow us to contract birational extremal rays. 

\begin{proof}[Proof of Theorem \ref{lc-basepoint-free} when $Z$ is a point] \

\textbf{Step 1}: Set-up.

We first prove that if $(Y,B_Y)$ is a projective $\Q$-factorial dlt $3$-fold pair over $k$ with $\Q$-boundary $B_Y$, such that $K_Y+B_Y$ is big and nef, then $K_Y+B_Y$ is semi-ample.  We spend most of the proof on this case, and then extend to general log canonical pairs with $\R$-boundaries in the final step.  We may assume $\rddown{B_Y}\neq 0$ as otherwise $(Y,B_Y)$ is klt and we can apply the base point free theorem \cite[1.4]{birkar_p}.  If it were the case that  $\E(K_Y+B_Y)\subseteq\rddown{B_Y}$, we would be able to apply Theorem \ref{KeelE} together with Theorem \ref{restrict} to conclude the statement.  However, this need not be true, so we proceed by modifying $Y$ to reach a situation where it holds.  For this reason, we would like to remove $K_Y+B_Y$-trivial curves which are not contained in $\rddown{B_Y}$.

\textbf{Step 2}: Contract $K_{Y}+B_{Y}$-trivial curves which intersect $\rddown{B_{Y}}$ positively.

$K_Y+B_Y$ is big by assumption, so let $\epsilon$ be sufficiently small that $K_Y+B_Y-\epsilon\rddown{B_Y}$ is also big, and so any $K_Y+B_Y-\epsilon\rddown{B_Y}$-MMP terminates by Theorem \ref{EffTerm}.  If we run a $K_Y+B_Y-\epsilon\rddown{B_Y}$-MMP with scaling of $\rddown{B_Y}$, by Lemma \ref{lem:LMMP-scaling-num-trivial}, we can replace $\epsilon$ by a smaller number to assume that the LMMP contracts only $K_Y+B_Y$-trivial extremal rays.

  As any contraction $Y_i\to Z_i$ in this LMMP is $K_{Y_i}+B_{Y_i}$-trivial, we claim that at each step $K_{Y_i}+B_{Y_i}$ pulls back from some $\Q$-Cartier divisor on $Z_i$.  
This follows from the klt cone and base point free theorems (see \cite[3.7(4)]{kollar_birational_1998}).

Therefore this LMMP results in a model $Y'$ on which the birational transform $K_{Y'}+B_{Y'}$ is semi-ample if and only if $K_Y+B_Y$ is semi-ample.  
$(Y',B_{Y'})$ is $\Q$-factorial and log canonical but may no longer be dlt.  By our application of Lemma \ref{lem:LMMP-scaling-num-trivial}, we have that $K_{Y'}+B_{Y'}-\epsilon'\rddown{B_{Y'}}$ is nef for all $\epsilon'\in[0,\epsilon]$.  Thus any $K_{Y'}+B_{Y'}$-trivial curve cannot intersect $\rddown{B_{Y'}}$ positively: it is therefore forced to either be disjoint from $\rddown{B_{Y'}}$ or completely contained in $\rddown{B_{Y'}}$.  

\textbf{Step 3}: Contract $K_{Y'}+B_{Y'}$-trivial curves not contained in $\rddown{B_{Y'}}$.

The underlying variety $Y'$ is klt and $\Q$-factorial because it was formed by running the LMMP from a klt $\Q$-factorial pair, so by Proposition \ref{polytope}(1), we may replace $\epsilon$ to be sufficiently small that any $K_{Y'}+B_{Y'}-\epsilon\rddown{B_{Y'}}$-trivial curve is also $K_{Y'}+B_{Y'}$-trivial.
Apply the base point free theorem \cite[1.4]{birkar_p} on the klt $K_{Y'}+B_{Y'}-\epsilon\rddown{B_{Y'}}$ to produce another birational $K_{Y'}+B_{Y'}$-trivial  contraction  $f:Y'\to Y''$ such that $K_{Y'}+B_{Y'}-\epsilon\rddown{B_{Y'}}\lin_\Q f^*A$ for some ample $\Q$-divisor $A$.  
Notice that if $0<\epsilon'<\epsilon$,  $K_{Y'}+B_{Y'}-\epsilon'\rddown{B_{Y'}}$ is also semi-ample and its associated morphism contracts the same curves as $f$.  Therefore $K_{Y'}+B_{Y'}-\epsilon'\rddown{B_{Y'}}$ also pulls back from $Y''$.

This implies $\rddown{B_{Y'}}$ pulls back from a $\Q$-Cartier divisor $D$ on $Y''$, and so does $K_{Y'}+B_{Y'}$.  Let $B_{Y''}$ be the birational transform of $B_{Y'}$ on $Y''$, so that $K_{Y'}+B_{Y'}=f^*(K_{Y''}+B_{Y''})$ and $D=f_*\rddown{B_{Y'}}=\rddown{B_{Y''}}$.

We get a new log canonical pair $(Y'',B_{Y''})$ such that $K_{Y''}+B_{Y''}\lin_\Q A+\epsilon D$ for ample $A$ and effective $D$ as above.  This implies that $\E(K_{Y''}+B_{Y''})\subset D=\rddown{B_Y''}$, and we also know that $K_{Y''}+B_{Y''}$ is semi-ample if and only if $K_{Y'}+B_{Y'}$ is.  However $(Y'',B_{Y''})$ need not be either dlt or $\Q$-factorial in general, so we still cannot apply Theorem \ref{restrict}.

\textbf{Step 4}:  Construct a model where Keel's theorem applies.

Let $(Y''',B_{Y'''})$ be a $\Q$-factorial dlt model of $(Y'',B_{Y''})$ with morphism $g:Y'''\to Y''$.  We claim that every irreducible component of $\E(K_{Y'''}+B_{Y'''})$ is either contained within $\rddown{B_{Y'''}}$ or is completely disjoint from it.  First note that $\LCS(Y'',B_{Y''})=\rddown{B_{Y''}}$, because $\rddown{B_{Y'}}$ pulls back from $\rddown{B_{Y''}}$, and $\LCS(Y',B_{Y'})= \rddown{B_{Y'}}$ because $(Y',B_{Y'}-\epsilon\rddown{B_{Y'}})$ is klt.

To complete the proof of the claim, suppose $V$ is an irreducible component of $\E(K_{Y'''}+B_{Y'''})$ and first assume that it is $2$-dimensional.  If $V$ is contracted over $Y''$ then it is in $\rddown{B_{Y'''}}$ by definition of the crepant dlt $\Q$-factorial modification.  
If it is not contracted over $Y''$ then its birational transform on $Y$ is in $\E(K_{Y''}+B_{Y''})\subset \rddown{B_{Y''}}$. 
So we may assume $V$ is $1$-dimensional.  $g^*\rddown{B_{Y''}}$ is an effective $\Q$-Cartier divisor on $Y'''$ with support equal to $\rddown{B_{Y'''}}$ (because $\LCS(Y'',B_{Y''})=\rddown{B_{Y''}}$).  If $V$ is contracted over $Y''$ then by the projection formula $V\cdot g^*\rddown{B_{Y''}}=0$.  If it is not contracted over $Y''$ then its birational transform is again contained within $\E(K_{Y''}+B_{Y''})$ and hence within $\rddown{B_{Y''}}$.  Either way this implies that either $V$ is contained within $\rddown{B_{Y'''}}$ or it is completely disjoint from it.  This completes the proof of the claim.

We may now apply Keel's theorem (Proposition \ref{KeelE}) to $(Y''',B_{Y'''})$.  Any connected component of $\E(K_{Y'''}+B_{Y'''})$ is either contained within $\rddown{B_{Y'''}}$ or is completely disjoint from it.  In this first case $K_{Y'''}+B_{Y'''}$ is semi-ample when restricted to this connected component by Theorem \ref{restrict}.  In the second it is semi-ample because in a neighbourhood of the component, $K_{Y'''}+B_{Y'''}$ is equal to $$K_{Y'''}+B_{Y'''}-\epsilon g^*\rddown{B_{Y''}}=g^*(K_{Y''}+B_{Y''}-\epsilon\rddown{B_{Y''}})$$ which is nef and big, and the pair $(Y''',B_{Y'''}-\epsilon g^*\rddown{B_{Y''}})$ is klt because $\Supp g^*\rddown{B_{Y''}}=\rddown{B_{Y'''}}$, so we may apply base point freeness in the klt case (\cite[1.4]{birkar_p}).

\textbf{Step 5} Log canonical pairs with $\R$-boundaries.

So far we have proved Theorem \ref{lc-basepoint-free} for $\Q$-factorial dlt pairs with $\Q$-boundaries.  Suppose now that $(X,B)$ is as in the statement, i.e. log canonical with $\R$-boundary $B$.  A crepant $\Q$-factorial dlt model $\phi:Y\to X$ exists by Proposition \ref{dlt}, and let $B_Y$ be the dlt $\R$-boundary defined by $K_Y+B_Y=\phi^*(K_X+B)$.  Let $V$ be the $\R$-vector space of Weil divisors spanned by the components of $B_Y$, and define $\mathcal{L}$ to be the rational polytope from Propostion \ref{polytope}.  Apply Proposition \ref{polytope}(2) with the family of extremal rays equal to all extremal rays of $\overline{NE}(X)$.  We get a smaller rational polytope $\mathcal{P}$ containing $B_Y$ such that $K_Y+\Delta$ is nef for all $\Delta\in\mathcal{P}$.  Let the vertices of $\mathcal{P}$ be $B_1,...,B_n$.  We may shrink $\mathcal{P}$ around $B_Y$ to assume that $K_Y+B_i$ is big for all $i$.  By the case we have already proved, $K_Y+B_i$ is semi-ample for each $i$.  But we can write $B_Y$ as some $\R$-linear combination of the $B_i$ with positive coefficients, and so this implies that $K_Y+B_Y$ is also semi-ample.  This in turn implies that $K_X+B$ is semi-ample.

\end{proof}

\begin{proof}[Proof of Theorem \ref{lc-basepoint-free} in relative case]
We now have a projective contraction $f:X\to Z$.  Let $A$ be an ample divisor on $Z$.  Because $K_X+B$ is big$/Z$, there exits $n\gg0$ such that $K_X+B+nf^*A$ is big.  
Now using Theorem \ref{lc-cone}(2), perhaps after increasing $n$, $K_X+B+nf^*A$ is also globally nef and positive on every curve not contracted$/Z$.  Our ground field is algebraically closed, so it is $F$-finite, infinite and perfect.  By the results of \cite{tanaka_semiample_perturbations} we can find $A'\lin_{\R}nf^*A$ such that $(X,B+A')$ is log canonical.  We may now apply the global case.
\end{proof}

\subsection{Big boundary divisors}

Now we move on to Theorem \ref{lc-mori-contraction}.  The proof follows that of Theorem \ref{lc-basepoint-free} in outline, but differs in detail as we must deal with non-birational morphisms and non pseudo-effective log divisors.

\begin{proof}[Proof of Theorem \ref{lc-mori-contraction} when $Z$ is a point]\ 

\textbf{Step 1}: Set-up.

As before, we first prove the theorem in the $\Q$-boundary, $\Q$-factorial dlt case.  To this end, assume that $(Y,B_Y)$ is a projective $\Q$-factorial dlt $3$-fold over $k$ with $\Q$-boundary $B_Y$.  Assume also that $A_Y$ is a big and semi-ample $\Q$-divisor, such that $K_Y+B_Y+A_Y$ is nef.  We prove that $K_Y+B_Y+A_Y$ is semi-ample.  By \cite{tanaka_semiample_perturbations} we may replace $A_Y$ to assume that $(Y,B_Y+A_Y)$ is log canonical.  We may also assume that $\rddown{B_Y}\neq 0$ by the base-point free theorem.

\textbf{Step 2}:  Run a $K_Y+B_Y+A_Y-\epsilon\rddown{B_Y}$-MMP with scaling of $\rddown{B_Y}$.

In fact, this step consists of showing that there is a way to choose such an LMMP which terminates, and that by taking $\epsilon$ sufficiently small we may assume that each contraction is $K_Y+B_Y+A_Y$-trivial.

First note that by Lemma \ref{lem:LMMP-scaling-num-trivial}, whenever the LMMP with scaling attempts to contract an extremal ray which is not $K_Y+B_Y+A_Y$-trivial, we may replace $\epsilon$ with a smaller number so that we terminate instead.  Lemma \ref{lem:LMMP-scaling-num-trivial} also tells us that our LMMP is also a $K_Y+B_Y+A_Y-\epsilon'\rddown{B_Y}$-MMP for any $\epsilon'\in (0,\epsilon]$.  This means that we are free to replace $\epsilon$ by a smaller number at any point during the LMMP without affecting the validity of the previous steps.

We now show that we can choose such an LMMP which terminates.  As a first step we claim that there is some choice of $K_Y+B_Y+A_Y-\epsilon\rddown{B_Y}$-MMP with scaling of $\rddown{B_Y}$ with the following properties:
\begin{itemize}
\item{Every step is $K_Y+B_Y+A_Y$-trivial.}
\item{It reaches (but may not terminate on) a model $\tilde{Y}$ such that:}
\item{No $K_{\tilde{Y}}+B_{\tilde{Y}}+A_{\tilde{Y}}-\epsilon\rddown{B_{\tilde{Y}}}$-MMP with scaling of $\rddown{B_{\tilde{Y}}}$ which contracts only $K_{\tilde{Y}}+B_{\tilde{Y}}+A_{\tilde{Y}}$-trivial rays will ever contain a divisorial contraction.}
\end{itemize}
 Suppose otherwise.  Then given any model reachable via a $K_Y+B_Y+A_Y-\epsilon\rddown{B_Y}$-MMP which contracts only $K_Y+B_Y+A_Y$-trivial extremal rays, it is possible to find a way to continue contracting only $K_Y+B_Y+A_Y$-trivial rays and reach a divisorial contraction.  By induction this produces an infinite sequence of divisorial contractions, which is impossible.  Therefore there must exist a model $\tilde{Y}$ as described. Thus, however we continue to run our LMMP with scaling from $\tilde{Y}$ then we may assume that every step is a flip.

Let $H_{\tilde{Y}}$ be an ample $\Q$-Cartier divisor on $\tilde{Y}$ such that $K_{\tilde{Y}}+B_{\tilde{Y}}+A_{\tilde{Y}}-\epsilon\rddown{B_{\tilde{Y}}}+H_{\tilde{Y}}$ is nef.  We may apply \cite{tanaka_semiample_perturbations} to assume that $(\tilde{Y},B_{\tilde{Y}}+A_{\tilde{Y}}+H_{\tilde{Y}})$ is log canonical.  Apply Proposition \ref{polytope} to $\tilde{Y}$ and the rational vector space of Weil divisors spanned by $B_{\tilde{Y}}+A_{\tilde{Y}}-\epsilon\rddown{B_{\tilde{Y}}}$, $B_{\tilde{Y}}+A_{\tilde{Y}}$ and $B_{\tilde{Y}}+A_{\tilde{Y}}+H_{\tilde{Y}}-\epsilon\rddown{B_{\tilde{Y}}}$.  Let $\mathcal{L}$ be the polytope and $\alpha$ and $\delta$ be the real numbers obtained in Proposition \ref{polytope}.  Note that by definition, each of $B_{\tilde{Y}}+A_{\tilde{Y}}-\epsilon\rddown{B_{\tilde{Y}}}$, $B_{\tilde{Y}}+A_{\tilde{Y}}$ and $B_{\tilde{Y}}+A_{\tilde{Y}}+H_{\tilde{Y}}-\epsilon\rddown{B_{\tilde{Y}}}$ lie within the polytope $\mathcal{L}$.   Choose $0<\lambda\ll 1$ and let $\mathcal{P}$ be the sub-polytope with vertices given by $B_{\tilde{Y}}+A_{\tilde{Y}}-\lambda\epsilon\rddown{B_{\tilde{Y}}}$, $B_{\tilde{Y}}+A_{\tilde{Y}}$ and $B_{\tilde{Y}}+A_{\tilde{Y}}+\lambda(H_{\tilde{Y}}-\epsilon\rddown{B_{\tilde{Y}}})$.  By taking $\lambda$ sufficiently small we may assume that $||\Delta-(B_{\tilde{Y}}+A_{\tilde{Y}})||<\delta$ for any $\Delta$ in $\mathcal{P}$ where $||\cdot ||$ is as in Proposition \ref{polytope}.  Note that $K_{\tilde{Y}}+B_{\tilde{Y}}+A_{\tilde{Y}}+\lambda(H_{\tilde{Y}}-\epsilon\rddown{B_{\tilde{Y}}})$ is nef, as it is a linear combination of two nef divisors, and we are happy to replace $\epsilon$ by $\lambda\epsilon$ for the purposes of our LMMP.

  Now run a $(\tilde{Y},B_{\tilde{Y}}+A_{\tilde{Y}}-\lambda\epsilon\rddown{B_{\tilde{Y}}})$-MMP with scaling of $\lambda H_{\tilde{Y}}$.  At every stage, so long as we contract no divisor, Proposition \ref{polytope} ensures that every extremal ray contracted is also $K_{\tilde{Y}}+B_{\tilde{Y}}+A_{\tilde{Y}}$-trivial.  This means that this LMMP is also a $K_{\tilde{Y}}+B_{\tilde{Y}}+A_{\tilde{Y}}-\lambda\epsilon\rddown{B_{\tilde{Y}}}$-MMP with scaling of $\rddown{B_{\tilde{Y}}}$, and hence also a  $K_{\tilde{Y}}+B_{\tilde{Y}}+A_{\tilde{Y}}-\epsilon\rddown{B_{\tilde{Y}}}$-MMP with scaling of $\rddown{B_{\tilde{Y}}}$ and so by the construction of $\tilde{Y}$ we indeed never contract a divisor.  This LMMP terminates by klt termination with scaling \cite[1.5]{birkar_waldron} on some model $Y'$, which is either a log minimal model or a Mori fibre space.

\textbf{Step 3}: Case where $(Y',B_{Y'}+A_{Y'}-\epsilon\rddown{B_{Y'}})$ is a Mori fibre space.

Suppose $Y'$ has $K_{Y'}+B_{Y'}+A_{Y'}-\epsilon\rddown{B_{Y'}}$-Mori fibre space structure $g:Y'\to V$, where $V$ is normal.  By construction, $g$ is $K_{Y'}+B_{Y'}+A_{Y'}$-trivial, so the fibres of $g$ intersect $\rddown{B_{Y'}}$ positively and  there is some component $V'$ of $\rddown{B_{Y'}}$ which is surjective to $V$.  By \cite[3.7(4)]{kollar_birational_1998}, $K_{Y'}+B_{Y'}+A_{Y'}$ pulls back from some $\Q$-Cartier divisor $D$ on $V$.  Using Proposition \ref{dlt}, let $\psi:Y''\to Y'$ be a crepant $\Q$-factorial dlt model for $(Y',B_{Y'}+A_{Y'})$, with $V''$ the birational transform of $V'$ on $Y''$. $\psi^*(K_{Y'}+B_{Y'}+A_{Y'})|_{V''}=((g\circ \psi)|_{V''})^*D$, and by Theorem \ref{restrict} the left hand side is semi-ample.  Thus $D$ is semi-ample because $g|_{V'}:V'\to V$ is a surjective projective morphism to a normal variety (see \cite[2.10]{keel_basepoint_1999}).  Thus we are done in this case.

\textbf{Step 4}: Case where $(Y',B_{Y'}+A_{Y'}-\epsilon\rddown{B_{Y'}})$ is a log minimal model.

For a given $\epsilon$, we may assume that $K_{Y'}+B_{Y'}+A_{Y'}-\epsilon\rddown{B_{Y'}}$ is not big by Theorem \ref{lc-basepoint-free}.  We may write $A_{Y'}\lin_\Q C+ E$ for some ample $C$ and effective $E$.  Choosing $\delta$ sufficiently small we may ensure that $$(Y',B_{Y'}-\epsilon\rddown{B_{Y'}}+(1-\delta)A_{Y'}+\delta E)$$ is klt.  The base point free theorem for klt pairs \cite[1.2]{birkar_waldron} now implies that $K_{Y'}+B_{Y'}+A_{Y'}-\epsilon\rddown{B_{Y'}}$ is semi-ample.  As in the proof of Theorem \ref{lc-basepoint-free}, by applying this for various values of $\epsilon$ we obtain a contraction $f:Y'\to V$ satisfying.  
$$K_{Y'}+B_{Y'}+A_{Y'}-\epsilon\rddown{B_{Y'}}\lin_\Q f^*H$$ 
$$K_{Y'}+B_{Y'}+A_{Y'}\lin_\Q f^*(H+\epsilon D)$$
for some $H$ ample on $V$ and $D\geq 0$.  
We show that $H+\epsilon D$ is semi-ample.  If $V$ has dimension less than $2$ this is obvious, so we may assume it has dimension $2$.  $\E(H+\epsilon D)\subset\Supp(D)$, and we claim that $f^{-1}(\Supp (D))=\rddown{B_{Y'}}$.  Suppose there is a $1$-dimensional component $\Gamma$ of $f^{-1}(\Supp(D))$.  $\Gamma$ must be contracted by $f$ because $f$ has connected fibres, and so $\Gamma\cdot \rddown{B_{Y'}}=0$ by the projection formula.  Therefore $\Gamma$ is either contained in $\rddown{B_{Y'}}$ or is completely disjoint from it, which contradicts the connectedness of the fibres of $f$.   Let $g:Y''\to Y'$ be a crepant $\Q$-factorial dlt modification of $(Y',B_{Y'}+A_{Y'})$ with $K_{Y''}+\Delta=g^*(K_{Y'}+B_{Y'}+A_{Y'})$.  As $Y'$ is $\Q$-factorial and $\LCS(Y',B_{Y'}+A_{Y'})=\rddown{B_{Y'}}$, $g^{-1}(\rddown{B_{Y'}})=\rddown{\Delta}$. Theorem \ref{restrict} implies that $(K_{Y''}+\Delta)|_{(f\circ g)^{-1}D}$ is semi-ample.  We may now apply the semi-ampleness criterion \cite[7.1]{birkar_waldron}, derived from Keel's theorem to deduce that $H+\epsilon D$ is semi-ample. 

\textbf{Step 5}: Log canonical pairs with $\R$-boundaries and $\R$-Cartier $A$.

We now work with the log canonicial pair $(X,B)$ with $\R$-boundary $B$ and $\R$-Cartier $A$ from the statement.  Let $\phi:Y\to X$ be a crepant $\Q$-factorial dlt modification of $(X,B)$ which exists by Proposition \ref{dlt}.  We claim that we may replace $A$ by some $A'\lin_\R A$ to assume that $(X,B+A)$ is also log canonical and that $\phi$ is a crepant $\Q$-factorial dlt modification of $(X,B+A)$.  To see this, note that by \cite{tanaka_semiample_perturbations} there is some $A''\lin_\R 2A$ such that $(X,B+A'')$ is log canonical.  Set $A'=\frac{1}{2}A''$.  Any log canonical place of $(X,B+A')$ is then also a log canonical place of $(X,B)$.

We also wish to assume that $A$ is a $\Q$-divisor.  To this end, write $A=\sum a_iA_i$ where $0<a_i\in\R$ and $A_i$ are semi-ample $\Q$-divisors.  Let $a_i=a_i^1+a_i^2$ for each $i$, where we freely choose $0<a_i^1\in\R$ and $0<a_i^2\in \Q$.  Define $A_{\R}=\sum a_i^1A_i$ and $A_{\Q}=\sum a_i^2A_i$.  By choosing $a_i^1$ sufficiently small, we may assume that $A_{\Q}$ is big.  Now we may replace $B$ with $B+A_{\R}$ and $A$ with $A_{\Q}$ to assume $A$ is $\Q$-Cartier.

Define $B_Y$ to be the dlt $\R$-boundary satisfying $K_Y+B_Y=\phi^*(K_X+B)$, and let $A_Y=\phi^*A$ (which is also big and semi-ample).  It is now enough to show that $K_Y+B_Y+A_Y$ is semi-ample when $(Y,B_Y)$ is $\Q$-factorial dlt but $B_Y$ may be an $\R$-boundary.

Let $V'$ be the $\R$-vector space of $\R$-Weil divisors generated by the components of $B_Y$.  Now let $V=V'+A_Y$.  This is an affine space of Weil divisors, so let $\mathcal{L}$ be as defined in Proposition \ref{polytope}.  Let $\mathcal{P}\subset\mathcal{L}$ be those $\Delta\in\mathcal{L}$ such that $(K_Y+\Delta)\cdot R\geq 0$ for all extremal rays $R$ of $\overline{NE}(X)$.  This is a rational polytope by Proposition \ref{polytope}(2).  Label the vertices $B_1+A_Y,...,B_n+A_Y$.  For each $i$, $K_{Y}+B_i+A_Y$ is nef and so we may apply the version of Theorem \ref{lc-mori-contraction} for $\Q$-boundaries to deduce that $K_Y+B_i+A_Y$ is semi-ample.  But as $K_Y+B_Y+A_Y$ can be written as a linear combination of $K_Y+B_i+A_Y$ with positive coefficients, this is also semi-ample.

\end{proof}

\begin{proof}[Proof of Theorem \ref{lc-mori-contraction} in relative case]\ 
 
As in the relative case of Theorem \ref{lc-basepoint-free}, if we let $H$ be the pullback to $X$ of a sufficiently ample divisor on $Z$, $A+H$ is big and semi-ample and by Theorem \ref{lc-cone} $K_X+B+A+H$ is globally nef.  By \cite{tanaka_semiample_perturbations} we may replace $H$ up to $\Q$-linear equivalence so that  $(X,B+A+H)$ is log canonical.  Now we may apply the global case.

\end{proof}

\section{The LMMP}\label{LMMP}

Next we apply Theorem \ref{lc-basepoint-free} to construct flips for log canonical pairs.

\begin{proof}[Proof of Corollary \ref{lcflips}]

Let $(X,B)$ be a projective log canonical $3$-fold pair with flipping contraction $f:X\to Z$.  

$(X,B)$ has a $\Q$-factorial dlt model $g:Y\to X$ where $K_Y+B_Y=g^*(K_X+B)$ by Proposition \ref{dlt}.  Run an LMMP$/Z$ for $(Y,B_Y)$.  If $A$ is an ample divisor on $Z$, $K_Y+B_Y+ng^*f^*A$ is big for $n\gg0$ and so the LMMP terminates on a log minimal model $(Y^+,B_{Y^+})$ by Theorem \ref{EffTerm}.  By Theorem \ref{lc-basepoint-free}, $K_{Y^+}+B_{Y^+}$ is semi-ample$/Z$.  It therefore has a log canonical model $(X^+,B^+)/Z$, which we show is the flip of $(X,B)/Z$.
It suffices to show that no divisors are contracted by $X^+\bir X$ as $K_{X^+}+B_{X^+}$ is ample$/Z$.

Let $\phi:W\to X$ and $\phi^+:W\to X^+$ be a common resolution. 
$$L=\phi^*(K_X+B)-\phi^{+*}(K_{X^+}+B_{X^+})$$ is anti-nef$/Z$ and so by the negativity lemma is effective.  Suppose there is a divisor $E_+$ which is contracted by $X^+\bir X$, and let $E_W$ be its birational transform on $W$.  $L$ must have coefficient zero in $E_W$, for any Weil divisor extracted by $Y\to X$ has log discrepancy $0$ with respect to $(X,B)$, and hence $E_+$ appears with coefficient $1$ in $B_{X^+}$.  Note that  $\phi^{+*}(K_{X^+}+B_{X^+})|_{E_W}$ is big and nef over $Z$, so $L$ intersects negatively with a general curve on $E_W$ which is contracted over $Z$.  There is a family of such curves, because $E_W$ is contracted over $X$ and hence over $Z$.  But this negative intersection is impossible if $L$ has coefficient zero in $E_W$, because $L$ is effective.
\end{proof}

Next we prove the contraction theorem:

\begin{proof}[Proof of Corollary \ref{contraction-theorem}]
Let $A$ be an ample $\R$-divisor on $X$.  For a sufficiently small rational $\epsilon>0$, $R$ is also $K_X+B+\epsilon A$ negative.  By Theorem \ref{lc-cone} there are only finitely many $K_X+B+\epsilon A$-negative extremal rays.  Therefore we may find an $\R$-divisor $H$ such that $R$ is the only $H$-negative extremal ray of  $\overline{NE}_1(X)$.  We may also assume that $A'=H-(K_X+B+\epsilon A)$ is ample (i.e. positive on all of $\overline{NE}_1(X)$).  So we may replace $A$ by $A'+\epsilon A$ such that $R$ is the only $K_X+B+A$-negative extremal ray.

Let $\lambda=\inf\{t:K_X+B+tA\mathrm{\ is\ nef}\}$ (so $\lambda>1$).  $K_X+B+\lambda A$ is positive wherever $K_X+B$ is non-negative, and also positive on every extremal $K_X+B$-negative extremal ray except $R$.  This means the extremal face $(K_X+B+\lambda A)_{=0}\cap\overline{NE}(X)$ must be $R$ itself.  By Theorem \ref{lc-cone} $R$ contains a curve.

By Theorem \ref{lc-mori-contraction}, $K_X+B+\lambda A$ is semi-ample, and so induces the contraction of $R$.
\end{proof}

\begin{proof}[Proof of Corollary \ref{log-minimal-models}]
This is a consequence of all our other results.   Given a log canonical pair $(X,B)$ such that $K_X+B$ is not nef, we may find a $K_X+B$-negative extremal ray using Theorem \ref{lc-cone}.  There is a projective contraction contracting the curves in this ray by Corollary \ref{contraction-theorem}.  If it is a flipping contraction, the flip exists by Corollary \ref{lcflips}.  Finally if $K_X+B\num M\geq 0$ the program terminates by Theorem \ref{EffTerm}.  Under the additional assumption the log minimal model is good by Theorem \ref{lc-basepoint-free}.
\end{proof}

\bibliography{bib}{}
\bibliographystyle{amsplain}
\end{document}